\documentclass{article}
\usepackage{authblk}
\usepackage{graphicx}%
\usepackage{multirow}%
\usepackage{amsmath,amssymb,amsfonts}%
\usepackage{amsthm}%
\usepackage{mathrsfs}%
\usepackage[title]{appendix}%
\usepackage{xcolor}%
\usepackage{textcomp}%
\usepackage{manyfoot}%
\usepackage{booktabs}%
\usepackage{algorithm}%
\usepackage{algorithmicx}%
\usepackage{algpseudocode}%
\usepackage{listings}%

\theoremstyle{thmstyleone}%
\newtheorem{theorem}{Theorem}
\newtheorem{proposition}[theorem]{Proposition}%
\newtheorem{cor}{Corollary}
\newtheorem{remark}{Remark}%
\newtheorem{example}{Example}%
\theoremstyle{thmstyletwo}%
\theoremstyle{thmstylethree}%

\begin{document}

\title{Toeplitz operators with symmetric, alternating and anti-symmetric separately radial symbols on the unit ball }

\author{Armando S\'anchez-Nungaray\footnote{Universidad Veracruzana. Email: \texttt{armsanchez@uv.mx}} \and Jos\'e Rosales-Ortega\footnote{Universidad de Costa Rica, Email: \texttt{jose.rosales@ucr.ac.cr}} \and Carlos Gonz\'alez-Flores\footnote{Instituto Polit\'ecnico Nacional, Email: \texttt{cfgonzalez@esimez.mx}} }


%
%
%
%
%

\maketitle

\abstract{We  consider symmetric separately radial (with corresponding group $S_n\rtimes \mathbb{T}^n$) and alternating separately radial (with corresponding group $A_n\rtimes \mathbb{T}^n$) symbols, as well as the associated Toeplitz operators  on the weighted  Bergman spaces on the unit ball on $\mathbb{C}^n$. Using a purely representation theoretic approach we obtain that the $C^*$-algebras generated  by each family of such Toeplitz operators is commutative. Furthermore, we show that the symmetric separately radial Toeplitz operators are more general than radial Toeplitz operators, i.e., every radial Toeplitz operator is a symmetric separately radial.}

\hspace{-0.5cm}\textbf{keywords}: Unit ball, Toeplitz operator, Symmetric separately radial symbols, Commutative $C^*$-algebras.

\hspace{-0.5cm}\textbf{MSC Classification:} 47B35, 32A36, 22E46, 32M15.

\section{Introduction}\label{sec1}

\label{sec1}
During recent years an important object of study, in operator theory, has been  the  Algebra of Toeplitz operators acting on Bergman spaces on the unit ball $\mathbb{B}^n$ on $\mathbb{C}^n$, see  \cite{Bauer, Quiroga4, Grudsky1, Grudsky2, Quiroga00}. Of particular interest among these operators are the so-called radial and separately radial Toeplitz operators because they are $n$-dimensional generalizations of \cite{Quiroga1}, where the Toeplitz operators with radial symbols on the unit disk generate a commutative $C^*$-algebra.
The search for commutative $C^*$-algebras generated by Toeplitz operators, $T_a$, in the cases where the symbol $a \in L^{\infty}(\mathbb{B}^n)$ is separately radial or radial has been a recent topic of research as one can find in \cite{Quiroga4,Grudsky1}.

The existence of these types of symbols on the unit ball on $\mathbb{C}^n$,  as well as in other bounded domains, as  invariant functions  with respect to symmetric   subgroups of the biholomorphism  group of the corresponding domain was used to generate  commutative  $C^*$-algebras of  Toeplitz operators as we can see in  \cite{Trieu, Quiroga0, Quiroga3}. These results were based on the relationship between the analytic and the representation theoretic approach  to the study of Toeplitz operators whose symbols have symmetries. A fundamental idea here was that if a symbol is invariant under a subgroup, then the corresponding Toeplitz operator intertwines the action of the subgroup (see Proposition 3.1 from \cite{ Quiroga2}).

The principal goal in this work is to explore the relationship  between the representation theory, as established in \cite{Quiroga2, Quiroga3}, and the commutativity of the $C^*$-algebra generated for the Toeplitz operators with symbols invariant under the groups $S_n \rtimes \mathbb{T}^n $ and $ \mathbb{T}^n \rtimes A_n$, where $\mathbb{T}^n$ stands  for the $n$-dimensional torus, $S_n$ stands for the symmetric group of permutations of the set $\{1,2,\ldots,n\}$, and $A_n$ is the alternating group of $S_n$. It is important to point out that the types of symbols obtained by introducing these new subgroups of $U(n)$, the unitary group,  generate new commutative algebras which have not been analyzed in current works on Toeplitz operators.

For the case of the symmetric separately radial symbols that correspond to the subgroup $S_n \rtimes \mathbb{T}^n $ of the biholomorphism group of the unit ball $\mathbb{B}^n$ on $\mathbb{C}^n$, we define a  unitary operator $R $ that allows us to exhibit the simultaneous diagonalization of the Toeplitz operators with symmetric separately radial symbols into multiplication operators on $\ell^2(\mathbb{Z}^n_+)$. The representation  theoretic approach  relies on Schur's lemma in the same way as found in \cite{Quiroga2} and it allows us to prove that the Toeplitz operators with symmetric separately radial symbols satisfy orthogonality relations. Also, the functions for the multiplication operators unitarily equivalent  to Toeplitz operators with symmetric separately radial symbols  are proved to be constant on the multi-indices that belong to the same $\mathcal{I}_{\iota}$. We refer to Section \ref{sec3}  for further details and for the moment $\mathcal{I}_{\iota}$ denotes an multi-indices set.
It is important to mention that symmetric separately radial symbols are more general than separately radial symbols. We present, in Section \ref{sec5}, an explicit  example to show this.

A corresponding study is made for alternating separately radial symbols, i.e. symbols that are $A_n \rtimes \mathbb{T}^n$-invariant, and for symbols that are antisymmetric separately radial.

We now describe the structure of the paper. In Section \ref{sec2} we explain some notations that are standard in the literature . We introduce various Lie groups  which  will be used throughout. In Section \ref{sec3},  for the case of symmetric separately radial symbols that correspond to the subgroup $S_n \rtimes \mathbb{T}^n$, we prove  the commutativity of the $C^*$-algebra of bounded operators on the weighted Bergman space of holomorphic functions  that intertwine the representation of $S_n\rtimes \mathbb{T}^n$. In Section \ref{sec4},  for the case of alternating separately radial symbols that correspond to the subgroup $A_n \rtimes \mathbb{T}^n$ we prove  an analogous result of commutativity. In Section \ref{sec5}  we prove that every Toeplitz operator with  symmetric separately radial  symbol is unitarily equivalent to a multiplication operator which is explicitly given. In Section \ref{sec6}  we prove that every Toeplitz operator with  alternating separately radial  symbol is unitarily equivalent to a multiplication operator which is explicitly given. In Section \ref{sec7} we define anti-symmetric  separately radial symbols and collect several properties of Toeplitz operators with such  symbols.

\section{Preliminaries}\label{sec2}
In 2006, Quiroga-Barranco (see \cite{Quiroga2}), using a representation theory approach, showed that the algebra of operators that intertwines a unitary representation over a Hilbert space is a commutative $C^*$-algebra. This was done by explicitly exhibiting an isomorphism with $\ell^{\infty}(J)$, where $J$  is some set of indices.The idea to prove this is to build for each $T$ that intertwines the representation  a function $\hat{\gamma}_T$.
Using the groups $\mathbb{T}^n$ and $U(n)$ appropriately, it performs explicit calculations with the Toeplitz operators generated by the separately radial and radial symbols, respectively.

This new technique of using representation theory in the study of Toeplitz operators acting on Bergman spaces has turned out to be shorter, more elementary, and more elucidating. It has also been adapted to other contexts such as projective spaces, Fock spaces and Cartan domains, among others. We refer to \cite{Bauer,Bauer2, dewage,Quiroga4, Esmerald1,  Grudsky1, Trieu, Grudsky2,  Quiroga0, Quiroga00,Quiroga3,  Quiroga7} for examples of research on Toeplitz operators with separately radial, radial symbols, and other types of symbols.

For the sake of completeness we present, in the next two remarks,  a summary of the ideas discussed by Quiroga-Barranco in \cite{Quiroga2}, see also \cite{Quiroga3, Quiroga6}.

 \begin{remark}
	Let $H$ be a Lie group and $\pi$ a unitary  representation on a Hilbert space $\mathcal{H}$. Suposse that $\mathcal{H}$ contains a dense subspace $V$ that can be algebraically decomposed as
	$$V= \sum_{i \in J} \mathcal{H}_j,$$
	where the subspaces $\mathcal{H}_j$ are mutually orthogonal, closed in $\mathcal{H}$ and irreducible $H$ invariant modules. Then if $H_{j_1}$ and $H_{j_2}$ are not isomorphic as $H$-modules for $j_1 \neq j_2$, then  the algebra of intertwining operators, denoted by $\mbox{End}_H(\mathcal{H})$, is commutative, and $\mathcal{H} =\bigoplus_{i \in J} \mathcal{H}_j$.
\end{remark}

\begin{remark}
	Choose $(e_l)_{l\in L}$ an orthonormal basis of $\mathcal{H}$ for which we have a disjoint union
	$$L=\bigcup_{j\in J}L_j,$$
	so that for every $j\in J$ the set $(e_l)_{l\in L_j}$ is an orthonormal basis for $\mathcal{H}_j$. We can consider the operator $R:\mathcal{H} \to \ell^2(\mathbb{N})$ and its adjoint operator $R^*$ such that every $T\in \mbox{End}_H(\mathcal{H})$ is unitarily equivalent to $RTR^*$, which is the multiplication operator on $\ell^2(L)$ given by the function $\gamma_T:L\to \mathbb{C}$ where $\gamma_T(l)=<T(e_l),e_l>$. Furthermore, the function $\gamma_T$ is constant  on $L_j$ for every $j \in J$ and so induces a function $\hat{\gamma}_T:J \to \mathbb{C}$ such that $\hat{\gamma}_T(j)=\gamma_T(l)$ whenever $i \in L_j$. In particular, the assignment $T \mapsto \hat{\gamma}_T$ is an isomorphism.
\end{remark}

Now we proceed to define the standard notation that will be used throughout this work.

Let  $\mathbb{B}^n=\{z=(z_1,\ldots,z_n)\in \mathbb{C}^n: \vert z \vert^2= \vert z_1 \vert^2 + \cdots + \vert z_n \vert^2 < 1 \}$ be the unit ball in $\mathbb{C}^n$. Using the convention from \cite{Zhu1}, given a multi-index $m=(m_1,\ldots,m_n) \in  \mathbb{Z}_+^n$ and $z=(z_1,\ldots,z_n)\in \mathbb{C}^n$ we have
\begin{eqnarray*}
	\vert m\vert &=& m_1 + m_2 + \cdots +m_n,   \\
	m ! &= & m_1!m_2!\cdots m_n!, \\
	z^{m} &=& z_1^{m_1}z_2^{m_2}\cdots z_n^{m_n}.
\end{eqnarray*}

For every $\alpha > -1$ we consider the weighted measure $$dv_{\alpha}(z) = c_{\alpha} \left( 1-\vert z \vert^2 \right)^{\alpha} \, dv(z),$$ where the constant $c_{\alpha}=\displaystyle  \frac{\Gamma(n+\alpha + 1)}{n! \Gamma(\alpha +1)}$ is chosen so that $v_{\alpha}(\mathbb{B}^n)=1$ and $dv(z)$ is the usual Lebesgue measure in $\mathbb{C}^n$. The weighted Bergman space $\mathcal{H}^2_{\alpha}(\mathbb{B}^n)$
is the subspace of holomorphic functions that lie in $L^2(\mathbb{B}^n, v_{\alpha})$.

For every $a\in L^{\infty}(\mathbb{B}^n)$ the Toeplitz operator on the weighted Bergman space $\mathcal{H}^2_{\alpha}(\mathbb{B}^n)$ is defined by $T_a:\mathcal{H}^2_{\alpha}(\mathbb{B}^n) \to \mathcal{H}^2_{\alpha}(\mathbb{B}^n)$ with
$$T_a(f)(z)=\int_{\mathbb{B}^n}a(w)f(w) (1-z\cdot \overline{w})^{-(n+\alpha + 1)}\, dv_{\alpha}(w). $$

In this case, $a$ is called the symbol of the Toeplitz operator $T_a$.

From now on we denote by $S_n$ the symmetric group of the set $\{1,\ldots, n\}$. Denote by $A_n$ the subgroup of $S_n$ formed by the even permutations.

Let $\mathbb{T}$ be the group of complex numbers with modulus equal to one. Let $\mathbb{T}^n = \mathbb{T}\times \cdots \times \mathbb{T}$.

The semi-direct group, $S_n \rtimes \mathbb{T}^n$ is a Lie group with operation given by $(\sigma, z)( \tau, w)= (\sigma \circ \tau, [\pi(\tau^{-1})z]w)$, where $\pi: S_n \to Aut(\mathbb{T}^n)$ is defined by the formula  $\pi(\sigma)(z)=(z_{\sigma^{-1}(1)}, \ldots, z_{\sigma^{-1}(n)})$. In a completely analogous way, we define
$A_n \rtimes \mathbb{T}^{n}$.

The action of $S_n \rtimes \mathbb{T}^n$ on $\mathbb{B}^n$ is given by
$$(\sigma, t) z=(t_1z_{\sigma(1)},\ldots, t_nz_{\sigma(n)}).$$

Let $N$ be a subgroup of $U(n)$. We define, $\mathcal{A}_N$, as the set of symbols $a\in L^{\infty}(\mathbb{B}^n)$ that satisfies $a(hz)=a(z)$ for all $h \in N$, and $z\in \mathbb{B}^n $. We define $\mathcal{T}^N=< T_a: a\in \mathcal{A}_N>$, i.e, the algebra generated  by the Toeplitz operators with symbol $a \in \mathcal{A}_N$.

We also define the following sets of symbols and their respective algebras generated by Toeplitz operators: $\mathcal{A}_{\mathbb{T}^{n}}$  and $\mathcal{T}^{\mathbb{T}^{n}}=<T_a: a\in \mathcal{A}_{\mathbb{T}^{n}}>$; $\mathcal{A}_{S_n \rtimes \mathbb{T}^n}$  and $\mathcal{T}^{S_n \rtimes \mathbb{T}^n}=< T_a: a\in \mathcal{A}_{S_n \rtimes \mathbb{T}^n}>$; $\mathcal{A}_{A_n \rtimes \mathbb{T}^{n}}$  and $\mathcal{T}^{A_n \rtimes \mathbb{T}^{n}}=< T_a: a\in \mathcal{A}_{A_n \rtimes \mathbb{T}^{n}}>$.

We have, for every $\alpha > -1$, a unitary representation $\pi_{\alpha}: S_n \rtimes \mathbb{T}^n \to U(\mathcal{H}^2_{\alpha}(\mathbb{B}^n))$  given by
$$(\pi_{\alpha}(\sigma,w))(f)(z)=f((\sigma,w)^{-1}z).$$

The following result identifies the type of Toeplitz operators that intertwines the previous representation $\pi_{\alpha}$, that is $T_a(\pi_{\alpha}(h)f)= \pi_{\alpha}(h)T_a(f)$, where $h\in S_n \rtimes \mathbb{T}^n$, and $f\in \mathcal{H}^2_{\alpha}(\mathbb{B}^n)$. We denote by $\mbox{End}_{ S_n\rtimes \mathbb{T}^n } (\mathcal{H}^2_{\alpha}(\mathbb{B}^n))$ the algebra of operators $T:\mathcal{H}^2_{\alpha}(\mathbb{B}^n)\to \mathcal{H}^2_{\alpha}(\mathbb{B}^n)$  that intertwine the representation $\pi_{\alpha}$.

\begin{proposition}{\label{prop1}}
	If $a\in L^{\infty}(\mathbb{B}^n)$ satisfies $a\circ ( \sigma, z) = a$, for every $( \sigma, z) \in S_n \rtimes \mathbb{T}^n$, then, for every $\alpha > -1$, the Toeplitz operator $T_a$ is an element of $ \mbox{End}_{ S_n\rtimes \mathbb{T}^n } (\mathcal{H}^2_{\alpha}(\mathbb{B}^n))$.
\end{proposition}
\begin{remark}
The proof of  Proposition \ref{prop1} uses the fact that the measure $v_{\alpha}$ is U(n)-invariant and so $S_n \rtimes \mathbb{T}^n$-invariant as well, where $U(n)$ is the unitary group (see  \cite{Quiroga44, Quiroga2, Quiroga3}).
\end{remark}

\section{ $S_n \rtimes \mathbb{T}^n$-intertwining operators}{\label{sec3}}

In this section, we show that the algebra $\mbox{End}_{ S_n\rtimes \mathbb{T}^n } (\mathcal{H}^2_{\alpha}(\mathbb{B}^n))$  of bounded operators on $\mathcal{H}^2_{\alpha}(\mathbb{B}^n)$ that intertwine the representation of $S_n \rtimes \mathbb{T}^n$ is commutative by constructing it  as an algebra  of multiplication operators. Our main tool is to express the space $\mathcal{H}^2_{\alpha}(\mathbb{B}^n)$ as a direct sum of suitable Hilbert spaces. This is achieved by isotopic decomposition. For more details see \cite{Quiroga6}.

In what follows we will use the multi-index notation without further mention, as appear in previous section, for polynomials functions on $\mathbb{C}^n$.

Inspired by the methods used in \cite{Quiroga2,Quiroga3}, we define the set of indices
$$  \mathcal{I}=\{\iota=(\iota_1,\ldots,\iota_n)\in \mathbb{Z}_+^n: \iota_1\geq \iota_2\geq \cdots \geq \iota_n \geq 0  \}.$$

For each $\iota\in \mathcal{I}$, we also define the set:
$$ \mathcal{I}_\iota=\{ \sigma(\iota)=(\iota_{\sigma(1)}, \ldots, \iota_{\sigma(n)})\in \mathbb{Z}_+^n: \sigma \in S_n\}.$$

Let us denote by
$\mathcal{P}_\iota(\mathbb{C}^n)$ the space of polynomials
generated by the monomials of the form
$$ z^{\sigma(\iota)}= z_1^{\iota_{\sigma(1)}}\cdots z_n^{\iota_{\sigma(n)}},$$
where $\sigma \in S_n $.

This new set of indices, to the best our knowledge, has not been used before to define symbols that generate new Toeplitz operators as will be shown later.

For the sake of clarity, we present in the following table, the decomposition associated the first indices
in dimension $n=3$.

For $n \geq 4$ the description of this set of indices becomes a bit more cumbersome.

{\tiny \begin{table}[htp]
		\caption{}\label{<table1>}%
		\begin{tabular}{|c|c|c|c|}
			\hline
			Degree &Elements of $\mathcal{I}$ & $\mathcal{I}_{\iota}$& $\mathcal{P}_\iota(\mathbb{C}^3)$\\
			\hline
			0 & (0,0,0)&(0,0,0) &1\\
			\hline
			1 & (1,0,0)&(1,0,0),(0,1,0),(0,0,1) &$z_1,z_2,z_3$\\
			\hline
			2 & (2,0,0)&(2,0,0),(0,2,0),(0,0,2) &$z_1^2,z_2^2,z_3^2$\\
			\hline
			2 & (1,1,0)&(1,1,0),(1,0,1),(0,1,1) &$z_1z_2, z_1z_3, z_2z_3$\\
			\hline
			3 & (3,0,0)&(3,0,0),(0,3,0),(0,0,3) &$z_1^3,z_2^3,z_3^3$\\
			\hline
			3 & (2,1,0)&(2,1,0),(2,0,1),(1,2,0)&$z_1^2 z_2, z_1^2 z_3, z_2^2 z_1$\\
			& &(0,2,1), (1,0,2),(0,1,2)&$z_2^2 z_3, z_3^2 z_1, z_3^2 z_2$\\
			\hline
			3 & (1,1,1)&(1,1,1) &$z_1z_2z_3$\\
			\hline
			4 & (4,0,0)&(4,0,0),(0,4,0),(0,0,4) &$z_1^4,z_2^4,z_3^4$\\
			\hline
			$k_1+k_2+k_3$ & $(k_1,k_2,k_3)$&$(k_1,k_2,k_3),(k_1,k_3,k_2)$&$z_1^{k_1} z_2^{k_2} z_3^{k_3}, z_1^{k_1} z_2^{k_3} z_3^{k_2}$\\
			& & $(k_2,k_1,k_3),(k_2,k_3,k_1)$ &$ z_1^{k_2}z_2^{k_1}z_3^{k_3}, z_1^{k_2}z_2^{k_3}z_3^{k_1} $\\
			&  &$(k_3,k_1,k_2),(k_3,k_2,k_1)$&$z_1^{k_3} z_2^{k_1} z_3^{k_2}, z_1^{k_3} z_2^{k_2} z_3^{k_1}$\\
			\hline
			$2k_1+k_2$ & $(k_1,k_1,k_2)$ & $(k_1,k_1,k_2),(k_1,k_2,k_1)$ & $z_1^{k_1} z_2^{k_1} z_3^{k_2},z_1^{k_1} z_2^{k_2} z_3^{k_1}$ \\
			& & $(k_2,k_1,k_1)$ & $z_1^{k_2} z_2^{k_1} z_3^{k_1}$ \\
			\hline
			$k_1+2k_2$ & $(k_1,k_2,k_2)$ & $(k_1,k_2,k_2),(k_2,k_1,k_2)$ & $z_1^{k_1} z_2^{k_2} z_3^{k_2},z_1^{k_2} z_2^{k_1} z_3^{k_2}$\\
			& & $ (k_2,k_2,k_1)$ & $z_1^{k_2} z_2^{k_2} z_3^{k_1}$\\
			\hline
			$3k_1$ & $(k_1,k_1,k_1)$ & $(k_1,k_1,k_1)$ & $z_1^{k_1} z_2^{k_1} z_3^{k_1}$\\
			\hline
		\end{tabular}
		\label{default}
\end{table}}

It is well known that the space $\mathcal{P}(\mathbb{C}^n)$ is dense and $U(n)$-invariant in $\mathcal{H}^2_{\alpha}(\mathbb{B}^n)$ for every $\alpha > -1$ (see for example \cite{Nollan1}).

\begin{proposition}{\label{prop2}}
	The decomposition of  $\mathcal{P}(\mathbb{C}^n)$ into irreducible $S_n\rtimes \mathbb{T}^n$-modules is given by
	$$\mathcal{P}(\mathbb{C}^n)=\sum_{\iota\in \mathcal{I}} \mathcal{P}_\iota(\mathbb{C}^n)$$
	More precisely, for every $\iota\in \mathcal{I}$, the space $\mathcal{P}_\iota(\mathbb{C}^n)$ is an irreducible $S_n\rtimes \mathbb{T}^n$
	-submodule.
	Moreover, for $\iota_1,\iota_2\in \mathcal{I}$ we have $\mathcal{P}_{\iota_1}(\mathbb{C}^n) \ncong \mathcal{P}_{\iota_2}(\mathbb{C}^n) $
	as $S_n\rtimes \mathbb{T}^n$-modules and $\mathcal{P}_{\iota_1}(\mathbb{C}^n) \perp \mathcal{P}_{\iota_2}(\mathbb{C}^n) $
	whenever $\iota_1\neq \iota_2$.

	In particular, for every $\alpha > -1$ we have
	$$\mathcal{H}^2_{\alpha}(\mathbb{B}^n)=\bigoplus_{\iota\in \mathcal{I}} \mathcal{P}_{\iota}(\mathbb{C}^n)$$
	as an orthogonal direct sum of Hilbert spaces that yields the decomposition of $\mathcal{H}^2_{\alpha}(\mathbb{B}^n)$
	into irreducible $S_n\rtimes \mathbb{T}^n$-modules.
\end{proposition}

\begin{proof}
	
	Note that for each monomial $z^m$ with $m\in \mathbb{Z}_+^n$ the action
	
	\begin{eqnarray}\label{Action on monomials}
		(\sigma,t)\cdot z^m=(\sigma^{-1}(t^{-1}z))^m&=&(t_{\sigma^{-1}(1)}^{-1}z_{\sigma^{-1}(1)} )^{m_1}\cdots (t_{\sigma^{-1}(n)}^{-1}z_{\sigma^{-1}(n)} )^{m_n}\\  \nonumber
		& =&t_{\sigma^{-1}(1)}^{-m_1}   \cdots t_{\sigma^{-1}(n)}^{-m_n} z_{\sigma^{-1}(1)}^{m_1} \cdots z_{\sigma^{-1}(n)}^{m_n}\\  \nonumber
		& =&t_{1}^{-m_{\sigma(1)}}   \cdots t_{n}^{-m_{\sigma(n)}} z_{1}^{m_{\sigma(1)}}   \cdots z_{n}^{m_{\sigma(n)}}\\ \nonumber&=&  t^{-\sigma(m)}z^{\sigma(m)}
	\end{eqnarray}
	where  $\sigma(m)=(m_{\sigma(1)},\ldots, m_{\sigma(n)})\in \mathbb{Z}_+^n$.
	
	From the above equation it is clear that the spaces $\mathcal{P}_{\iota}(\mathbb{C}^n)$ are $S_n\rtimes \mathbb{T}^n$-invariant. Such a sum yields a decomposition
	into irreducible $S_n\rtimes \mathbb{T}^n$-modules.

	Now, we consider $\iota_1,\iota_2\in \mathcal{I}$ with $\iota_1\neq \iota_2$,  $m\in \mathcal{I}_{\iota_1}$ and $m'\in \mathcal{I}_{\iota_2}$ then
	\begin{eqnarray*}
		\langle (\sigma,t)\cdot z^m, z^{m'}\rangle=t^{-\sigma(m)}\langle z^{\sigma(m)}, z^{m'}\rangle=0
	\end{eqnarray*}
	for each $(\sigma,t) \in S_n\rtimes \mathbb{T}^n$, since $\sigma(m)\in \mathcal{I}_{\iota_1}$ and $\mathcal{I}_{\iota_1}\cap \mathcal{I}_{\iota_2}= \emptyset$.
	
	Since the isomorphism class of an irreducible $S_n\rtimes \mathbb{T}^n$-module is determined by (\ref{Action on monomials}) (see \cite{Brocker}), it is clear  that $\mathcal{P}_{\iota_1}(\mathbb{C}^n) \ncong \mathcal{P}_{\iota_2}(\mathbb{C}^n) $
	as $S_n\rtimes \mathbb{T}^n$-modules when $\iota_1\neq \iota_2$.
	
	The result follows by using remark 2 of section \ref{sec2}.
	
\end{proof}

Proposition \ref{prop2} allows us to apply results  from  Section 2 in \cite{Quiroga2}.

For every $\alpha > -1$, it is  well known that the following functions form an   orthonormal basis of $\mathcal{H}^2_{\alpha}(\mathbb{B}^n)$
$$e_m(z)=\sqrt{\frac{\Gamma(n+\vert m\vert +\alpha +1 ) }{m!\Gamma(n +\alpha +1 )}}z^m,$$
for more details we can see \cite{Zhu1}.

In particular,  for each $m\in \mathcal{I}_{\iota}$ we have that $m!=\iota!$ and $\vert m\vert=\vert \iota\vert$, therefore we obtain
$$e_m(z)=\sqrt{\frac{\Gamma(n+\vert \iota\vert +\alpha +1 ) }{\iota!\Gamma(n +\alpha +1 )}}z^m.$$
In other words, the factor of normalization of the monomial in the space $\mathcal{P}_{\iota}(\mathbb{C}^n)$ is the same, i.e.,
\begin{eqnarray*}
	\langle  z^m, z^{m}\rangle=\langle z^{\iota}, z^{\iota}\rangle=\frac{\iota!\Gamma(n +\alpha +1 )}{\Gamma(n+\vert \iota\vert+\alpha +1 ) }.
\end{eqnarray*}

\begin{theorem}{\label{Teor3} }
	For every $\alpha > -1$, the algebra $\mbox{End}_{S_n\rtimes \mathbb{T}^n} (\mathcal{H}^2_{\alpha}(\mathbb{B}^n))$ is commutative.
	More precisely, with the above notation and for the unitary map
	\begin{eqnarray*}
		R :& \mathcal{H}^2_{\alpha}(\mathbb{B}^n)\longrightarrow \ell^2(\mathbb{Z}^n_+) \label{R}\\
		&R(f) = (\langle f , e_{m}\rangle)_{m\in\mathbb{Z}^n_+ }  ,
	\end{eqnarray*}
	every operator $T\in \mbox{End}_{S_n\rtimes \mathbb{T}^n} (\mathcal{H}^2_{\alpha}(\mathbb{B}^n))$ is unitarily equivalent to $RT R^*$, which is the
	multiplication operator on $\ell^2(\mathbb{Z}^n_+)$,
	by the function
	\begin{eqnarray*}
		&\gamma_T : \mathbb{Z}^n_+ \longrightarrow \mathbb{C}\\   &\gamma_T(m)=\langle T e_{m}, e_{m} \rangle
	\end{eqnarray*}
	Furthermore, let us choose for every $\iota\in \mathcal{I}$ a unitary vector $u_{\iota} \in  P_{\iota}(\mathbb{C}^n)$ and
	consider the function
	\begin{eqnarray*}
		&\hat{\gamma}_T : \mathcal{I} \longrightarrow \mathbb{C}\\   &\hat{\gamma}_T(\iota)=\langle T u_{\iota}, u_{\iota} \rangle
	\end{eqnarray*}
	In particular, we can consider
	$$u_{\iota}=e_{\iota}(z)=\sqrt{\frac{\Gamma(n+\vert \iota\vert +\alpha +1 ) }{\iota!\Gamma(n +\alpha +1 )}}z^{\iota}.$$
	Then, we have
	$\hat{\gamma}_T(\iota)=\gamma_T(m)$
	for every $m\in \mathcal{I}_{\iota}$.
	Moreover, $ \gamma_T(m)=\gamma_T(m')$ whenever  $m, m'\in \mathcal{I}_{\iota}$.
\end{theorem}

	The assertion involving $R$ follows directly from Proposition 2.2 in \cite{Quiroga2}.
	
	By the last claim  of that proposition, it also follows that for every $\iota \in \mathcal{I}$, the function $\gamma_T(m)$ is
	constant on the set of values $m\in \mathcal{I}_{\iota}$ for which $e_m\in \mathcal{P}_{\iota}(\mathbb{C}^n)$.

	On the other hand, as in the proof of Proposition 2.1 in \cite{Quiroga2} and by Proposition \ref{prop2},
	using Schur's Lemma we obtain  that $T_a$ acts by a scalar multiple on $\mathcal{P}_{\iota}(\mathbb{C}^n) $ for every $\iota \in \mathcal{I}$.

As a consequence of  Proposition \ref{prop2} and Theorem \ref{Teor3} we obtain the following result.
\begin{cor}
	With the above notation, the assignment
	$$T\longmapsto \hat{\gamma}_T $$
	defines an isomorphism of algebras
	$$\mbox{End}_{ S_n\rtimes \mathbb{T}^n } (\mathcal{H}^2_{\alpha}(\mathbb{B}^n)) \longrightarrow \ell^{\infty}(\mathcal{I})$$
\end{cor}

\section{ $A_n \rtimes \mathbb{T}^n$-intertwining operators}{\label{sec4}}

In this section, we show that the algebra $\mbox{End}_{ A_n\rtimes \mathbb{T}^n } (\mathcal{H}^2_{\alpha}(\mathbb{B}^n))$  of bounded operators on $\mathcal{H}^2_{\alpha}(\mathbb{B}^n)$ that intertwine the representation of $A_n \rtimes \mathbb{T}^n$ is commutative by constructing it  as an algebra  of multiplication operators. Our main tool is to express the space $\mathcal{H}^2_{\alpha}(\mathbb{B}^n)$ as a direct sum of appropriate Hilbert spaces. The latter is known as the isotypic decomposition, see \cite{Quiroga7}

For each $\iota=(\iota_1,\ldots \iota_n)\in \mathcal{I}$, we also define the subsets of $\mathcal{I}_\iota$ as follows
$$\mathcal{I}_\iota^+=\{ \sigma(\iota)=(\iota_{\sigma(1)}, \ldots, \iota_{\sigma(n)})\in \mathbb{Z}_+^n: \sigma \in A_n  \}$$
and
$$\mathcal{I}_\iota^-=\{ \sigma(\iota)=(\iota_{\sigma(1)}, \ldots, \iota_{\sigma(n)})\in \mathbb{Z}_+^n: \sigma \in S_n\setminus A_n \}$$
We can check the following relation of the above sets
\begin{eqnarray} 
	\mathcal{I}_\iota^+ \cap \mathcal{I}_\iota^-=\emptyset   & \mbox{ for } \iota_1>\cdots >\iota_n  \\ \label{Iiota eq}
	\mathcal{I}_\iota^+ = \mathcal{I}_\iota^- & \mbox{ otherwise}
\end{eqnarray}

Now, we introduce the subsets of $\mathcal{I}$ given by
\begin{eqnarray}\label{I0}
	\mathcal{I}^0 &=&\{ \iota=(\iota_1,\ldots,\iota_n)\in \mathcal{I}: \exists\,  l_1,l_2 \in \mathbb{Z}^+ \mbox{ such that }\iota_{l_1}=\iota_{l_2} \}\\ \label{Ic}
	\mathcal{I}^c & =&\{ \iota=(\iota_1,\ldots,\iota_n)\in \mathcal{I}: \iota_1>\cdots > \iota_n \}
\end{eqnarray}

Let us denote by
$\mathcal{P}^+_\iota(\mathbb{C}^n)$ and $\mathcal{P}^+_\iota(\mathbb{C}^n)$  the spaces of  polynomials
generated by all the monomials of the form
$$ z^{\sigma(\iota)}= z_1^{\iota_{\sigma(1)}}\cdots z_n^{\iota_{\sigma(n)}}  $$
where $\sigma$ is even and odd respectively.

Using the above relations we can easily obtain the following

\begin{eqnarray} \label{Poly I0}
	\mathcal{P}_\iota(\mathbb{C}^n)=\mathcal{P}^+_\iota(\mathbb{C}^n)= \mathcal{P}^-_\iota(\mathbb{C}^n)   & \mbox{ for }\iota \in \mathcal{I}^0  \\ \label{Poly Ic}
	\mathcal{P}^+_\iota(\mathbb{C}^n)\perp \mathcal{P}^-_\iota(\mathbb{C}^n)     & \mbox{ for }\iota \in \mathcal{I}^c
\end{eqnarray}

For the sake of clarity, we present in the following table, the decomposition associated the first indices
in dimension $n=3$ for the sets above described.

{\tiny \begin{table}[htp]
				\caption{}\label{<table2>}%
		\begin{tabular}{|c|c|c|c|c|c|}
			\hline
			Degree &Elements of $\mathcal{I}$ & $\mathcal{I}^+_{\iota}$&  $\mathcal{I}^{-}_{\iota}$&$\mathcal{P}^+_\iota(\mathbb{C}^3)$&  $\mathcal{P}^{-}_\iota(\mathbb{C}^3)$\\
			\hline
			\hline
			$k_1+k_2+k_3$ & $(k_1,k_2,k_3)$&   $(k_1,k_2,k_3)$&  $(k_1,k_3,k_2)$ &$z_1^{k_1} z_2^{k_2} z_3^{k_3}$&$ z_1^{k_1} z_2^{k_3} z_3^{k_2}$\\
			& & $(k_2,k_3,k_1)$   & $(k_2,k_1,k_3)$   &$ z_1^{k_2}z_2^{k_3}z_3^{k_1} $& $ z_1^{k_2}z_2^{k_1}z_3^{k_3} $\\
			&  &$(k_3,k_1,k_2)$  &   $(k_3,k_2,k_1)$  &$z_1^{k_3} z_2^{k_1} z_3^{k_2}$&$ z_1^{k_3} z_2^{k_2} z_3^{k_1}$\\
			\hline
			$2k_1+k_2$ & $(k_1,k_1,k_2)$ & $(k_1,k_1,k_2)$ & $(k_1,k_1,k_2)$  &$z_1^{k_1} z_2^{k_1} z_3^{k_2}$ &$z_1^{k_1} z_2^{k_1} z_3^{k_2}$\\
			&  & $(k_1,k_2,k_1)$ &  $(k_1,k_2,k_1)$  &$z_1^{k_1} z_2^{k_2} z_3^{k_1}$ & $z_1^{k_1} z_2^{k_2} z_3^{k_1}$\\
			& & $(k_2,k_1,k_1)$ & $(k_2,k_1,k_1)$ &$z_1^{k_2} z_2^{k_1} z_3^{k_1}$& $z_1^{k_2} z_2^{k_1} z_3^{k_1}$\\
			\hline
			$k_1+2k_2$ & $(k_1,k_2,k_2)$ & $(k_1,k_2,k_2)$ & $(k_1,k_2,k_2)$ & $z_1^{k_1} z_2^{k_2} z_3^{k_2}$& $z_1^{k_1} z_2^{k_2} z_3^{k_2}$\\
			& & $(k_2,k_1,k_2)$ & $(k_2,k_1,k_2)$ & $z_1^{k_2} z_2^{k_1} z_3^{k_2}$& $z_1^{k_2} z_2^{k_1} z_3^{k_2}$\\
			& & $ (k_2,k_2,k_1)$ & $ (k_2,k_2,k_1)$  & $z_1^{k_2} z_2^{k_2} z_3^{k_1}$&$z_1^{k_2} z_2^{k_2} z_3^{k_1}$\\
			\hline
			$3k_1$ & $(k_1,k_1,k_1)$ & $(k_1,k_1,k_1)$  & $(k_1,k_1,k_1)$ & $z_1^{k_1} z_2^{k_1} z_3^{k_1}$& $z_1^{k_1} z_2^{k_1} z_3^{k_1}$\\
			\hline
		\end{tabular}
	\footnotetext{It can be noted that the decomposition for $A_n$ in this table is finer than that carried out in table 1.}
		\label{default}
\end{table}}

\begin{proposition}{\label{prop5}}
	The decomposition of  $\mathcal{P}(\mathbb{C}^n)$ into irreducible $A_n\rtimes \mathbb{T}^n$-modules is given by
	$$\mathcal{P}(\mathbb{C}^n)=\left(\sum_{\iota\in \mathcal{I}^0} \mathcal{P}_\iota(\mathbb{C}^n)\right)+ \left(\sum_{\iota\in \mathcal{I}^c} \mathcal{P}^+_\iota(\mathbb{C}^n) \bigoplus \mathcal{P}^-_\iota(\mathbb{C}^n)\right) $$
	More precisely, we have the following statements
	\begin{enumerate}
		\item For every $\iota\in \mathcal{I}^0$, the space $\mathcal{P}_\iota(\mathbb{C}^n)$ is an irreducible $A_n\rtimes \mathbb{T}^n$
		-submodule.
		\item For every $\iota\in \mathcal{I}^c$, the spaces $\mathcal{P}^+_\iota(\mathbb{C}^n)$ and $\mathcal{P}^-_\iota(\mathbb{C}^n)$  are irreducible $A_n\rtimes \mathbb{T}^n$-submodules.
	\end{enumerate}
	
	Moreover, we obtain the following list of relations
	\begin{enumerate}
		\item For $\iota_1,\iota_2\in \mathcal{I}$, we have $\mathcal{P}^{\pm}_{\iota_1}(\mathbb{C}^n) \ncong \mathcal{P}^{\pm}_{\iota_2}(\mathbb{C}^n) $
		as $A_n\rtimes \mathbb{T}^n$-modules and $\mathcal{P}^{\pm}_{\iota_1}(\mathbb{C}^n) \perp \mathcal{P}^{\pm}_{\iota_2}(\mathbb{C}^n) $
		whenever $\iota_1\neq \iota_2$
		\item For $\iota\in \mathcal{I}^c$, we have $\mathcal{P}^{+}_{\iota}(\mathbb{C}^n) \ncong \mathcal{P}^{-}_{\iota}(\mathbb{C}^n) $
		as $A_n\rtimes \mathbb{T}^n$-modules \\and $\mathcal{P}^{+}_{\iota}(\mathbb{C}^n) \perp \mathcal{P}^{-}_{\iota}(\mathbb{C}^n) $
	\end{enumerate}

	In particular, for every $\alpha > -1$ we have
	$$\mathcal{H}^2_{\alpha}(\mathbb{B}^n)=\left(\bigoplus_{\iota\in \mathcal{I}^0} \mathcal{P}_\iota(\mathbb{C}^n)\right)\bigoplus \left(\bigoplus_{\iota\in \mathcal{I}^c} \mathcal{P}^+_\iota(\mathbb{C}^n)\bigoplus \mathcal{P}^-_\iota(\mathbb{C}^n)\right)$$
	as an orthogonal direct sum of Hilbert spaces that yields the decomposition of $\mathcal{H}^2_{\alpha}(\mathbb{B}^n)$
	into irreducible $A_n\rtimes \mathbb{T}^n$-modules.
\end{proposition}

\begin{proof}
	The first sum in the statement follows trivially from what was shown in Proposition \ref{prop2} and from the relations \ref{Poly I0} and \ref{Poly Ic}.

	For each monomial $z^m$ with $m\in \mathbb{Z}_+^n$  it was proved in Proposition \ref{prop2} that
	\begin{eqnarray*}
		(\sigma,t)\cdot z^m = t^{-\sigma(m)}z^{\sigma(m)},
	\end{eqnarray*}
	where  $\sigma(m)=(m_{\sigma(1)}\ldots m_{\sigma(n)})\in \mathbb{Z}_+^n$.This remains true regardless of the parity of the permutation $\sigma$, so for $\iota\in \mathcal{I}^c$, we have $\mathcal{P}^{+}_{\iota}(\mathbb{C}^n)$ and  $\mathcal{P}^{-}_{\iota}(\mathbb{C}^n) $
	are irreducible $A_n\rtimes \mathbb{T}^n$-modules. The same conclusion is obtained when $\iota\in \mathcal{I}^0$.
	
	We consider for $\iota_1 \neq \iota_2 \in \mathcal{I}^c$, $m\in \mathcal{I}_{\iota_1}$ and  $m^{\prime}\in \mathcal{I}_{\iota_2}$. Using well known results about orthogonality relations  (see \cite{Zhu1}) and  the relationship
	\begin{eqnarray*}
		\langle (\sigma,t)\cdot z^m, z^{m'}\rangle=t^{-\sigma(m)}\langle z^{\sigma(m)}, z^{m'}\rangle,
	\end{eqnarray*}
	the orthogonality of $\mathcal{P}^{+}_{\iota_1}(\mathbb{C}^n)$ and    $\mathcal{P}^{+}_{\iota_2}(\mathbb{C}^n) $ is immediate. The other cases are treated in the same way.
	
	Since the isomorphism class  of an irreducible $A_n \rtimes \mathbb{T}^n$-module is determined by its character, it follows that  for $\iota_1,\iota_2\in \mathcal{I}$, we have $\mathcal{P}^{\pm}_{\iota_1}(\mathbb{C}^n) \ncong \mathcal{P}^{\pm}_{\iota_2}(\mathbb{C}^n) $
	as $A_n\rtimes \mathbb{T}^n$-modules.
\end{proof}

Proposition \ref{prop5} allows us to apply the results from Section 2 of \cite{Quiroga2}. We  consider, for every $\alpha > -1$, the same Hilbert basis $(e_m)_{m \in \mathbb{Z}_+^n}$ defined in the previous section. Hence,  Proposition 2.2 in \cite{Quiroga3} and Proposition \ref{prop5} yield the following result.
\begin{theorem}{\label{6}}
	For every $\alpha > -1$, the algebra $\mbox{End}_{A_n\rtimes \mathbb{T}^n} (\mathcal{H}^2_{\alpha}(\mathbb{B}^n))$ is commutative.
	More precisely, with the above notation and for the unitary map
	\begin{eqnarray*}
		R :& \mathcal{H}^2_{\alpha}(\mathbb{B}^n)\longrightarrow \ell^2(\mathbb{Z}^n_+)\\
		&R(f) = (\langle f , e_{m}\rangle)_{m\in\mathbb{Z}^n_+}   ,
	\end{eqnarray*}
	every operator $T\in \mbox{End}_{A_n\rtimes \mathbb{T}^n} (\mathcal{H}^2_{\alpha}(\mathbb{B}^n))$ is unitarily equivalent to $RT R^*$ which is the
	multiplication operator on $\ell^2(\mathbb{Z}^n_+)$
	by the function
	\begin{eqnarray*}
		&\gamma_T : \mathbb{Z}^n_+ \longrightarrow \mathbb{C}\\   &\gamma_T(m)=\langle T e_{m}, e_{m} \rangle
	\end{eqnarray*}
	Furthermore, let us choose the following unitary vectors
	\begin{enumerate}
		\item For every $\iota\in \mathcal{I}^0$ a unitary vector $u_{\iota}=e_{\iota}(z) \in  P_{\iota}(\mathbb{C}^n)$.
		\item For every $\iota\in \mathcal{I}^c$ a unitary vector $u^+_{\iota}=e_{\iota}(z) \in  P_{\iota}^+(\mathbb{C}^n)$.
		\item For every $\iota\in \mathcal{I}^c$ a unitary vector $u^-_{\iota}=e_{\sigma(\iota)}(z) \in  P_{\iota}^-(\mathbb{C}^n)$ where $\sigma$ is odd.
	\end{enumerate}
	We consider the function
	\begin{eqnarray*}
		&\hat{\gamma}_T : \mathcal{I}^0 \sqcup \mathcal{I}^c_+\sqcup \mathcal{I}^c_- \longrightarrow \mathbb{C}\\
		&\hat{\gamma}_T(\iota)=\left\{
		\begin{array}{ll}
			\langle T u_{\iota}, u_{\iota} \rangle   & \mbox{ for } \iota \in \mathcal{I}^0 \\
			\langle T u^+_{\iota}, u^+_{\iota} \rangle   & \mbox{ for } \iota \in \mathcal{I}^c_+=\mathcal{I}^c \\
			\langle T u^-_{\iota}, u^-_{\iota} \rangle   & \mbox{ for } \iota \in \mathcal{I}^c_-=\mathcal{I}^c
		\end{array}    \right.
	\end{eqnarray*}
	Then,
	\begin{enumerate}
		\item For $\iota \in \mathcal{I}^0$ we have
		$\hat{\gamma}_T(\iota)=\gamma_T(m)$
		for every $m\in \mathcal{I}_{\iota}$, i.e.,
		$ \gamma_T(m)=\gamma_T(m')$ whenever  $m, m'\in \mathcal{I}_{\iota}$.
		\item For $\iota \in \mathcal{I}^c$ we have
		$\hat{\gamma}_T(\iota)=\gamma_T(m)$
		for every $m\in \mathcal{I}_{\iota}^+$, i.e.,
		$ \gamma_T(m)=\gamma_T(m')$ whenever  $m, m'\in \mathcal{I}_{\iota}^+$.
		\item Consider $\sigma$ an odd permutation. For $\iota \in \mathcal{I}^c$  we have
		$\hat{\gamma}_T(\sigma(\iota))=\gamma_T(m)$
		for every $m\in \mathcal{I}_{\iota}^-$, i.e.,
		$ \gamma_T(m)=\gamma_T(m')$ whenever  $m, m'\in \mathcal{I}_{\iota}^-$.
	\end{enumerate}
\end{theorem}

\begin{proof}
	The assertion involving $R$ follows directly from Proposition 2.2 in \cite{Quiroga2}.
	By the last claim of that proposition,  it follows that for every $\iota \in \mathcal{I}^0 \sqcup \mathcal{I}^C$, the function $\gamma_T$ is constant on the set of values $m\in \mathcal{I}_{\iota}$ for which $e_m \in  \mathcal{P}_{\iota}(\mathbb{C}^n) \sqcup \mathcal{P}_{\iota}^+(\mathbb{C}^n) \sqcup \mathcal{P}_{\iota}^-(\mathbb{C}^n)$.
	
	On the other hand, Schur's Lemma implies that $T$ acts by scalar multiplication on each $\mathcal{P}_{\iota}(\mathbb{C}^n)$, $\mathcal{P}_{\iota}^+(\mathbb{C}^n)$, $\mathcal{P}_{\iota}^-(\mathbb{C}^n)$ respectively. We choose for $\iota \in \mathcal{I}^0$  a unitary vector  $u_{\iota}\in P_{\iota}(\mathbb{C}^n) $. It is easy to  obtain that
	$\langle T u_{\iota}, u_{\iota} \rangle = \langle Te_{m}, e_{m} \rangle$,
	for each $m \in \mathcal{I}_{\iota}$. In a completely analogous way if $\iota \in \mathcal{I}^{c}$ then we choose unit vectors $u_{\iota}^+ \in \mathcal{P}_{\iota}^+(\mathbb{C}^n)$ and $u_{\iota}^- \in \mathcal{P}_{\iota}^-(\mathbb{C}^n)$ such that
	$$\langle T u_{\iota}^+,u_{\iota}^+ \rangle = \langle T e_m, e_m \rangle, \bigskip \langle T u_{\iota}^-,u_{\iota}^- \rangle = \langle T e_k, e_k \rangle$$
	for each $m \in \mathcal{I}_{\iota}^+$ and $k \in \mathcal{I}_{\iota}^-$, respectively.

\end{proof}

As a consequence of Proposition 2.2 in \cite{Quiroga2} and Theorem \ref{6} we obtain the following result.

\begin{cor}
	With the above notation, the assignment
	$$T\longmapsto \hat{\gamma}_T $$
	defines an isomorphism of algebras
	$$\mbox{End}_{ A_n\rtimes \mathbb{T}^n } (\mathcal{H}^2_{\alpha}(\mathbb{B}^n)) \longrightarrow \ell^{\infty}(\mathcal{I}^0 \sqcup \mathcal{I}^c_+\sqcup \mathcal{I}^c_-)$$
\end{cor}

\section{Toeplitz operators with symmetric separately radial symbols}{\label{sec5}}

In this section we consider new types of symbols resulting from the action of certain semidirect products of symmetric groups with $\mathbb{T}^n$ on $\mathcal{H}_{\alpha}^2(\mathbb{B}^n)$.

Following \cite{Quiroga2}, we define a  symmetric separately radial
function $a \in  L^{\infty}(\mathbb{B}^n)$ if it can be written as follows
\begin{eqnarray}\label{Symm:function}
	a(z) = a(\vert z_1\vert, \ldots ,\vert z_n\vert)
\end{eqnarray}
and satisfies that
\begin{eqnarray}\label{Symm:function}
	a(\sigma(z))= a(z) \mbox{ where } \sigma \in S_n
\end{eqnarray}
for almost every $z \in \mathbb{B}^n$.

In other words, the function $a$ is symmetric separately radial if and only if it is $S_n\rtimes \mathbb{T}^n$-invariant.

The set of functions $S_n\rtimes \mathbb{T}^n$-invariant will be denoted by $\mathcal{A}_{S_n\rtimes \mathbb{T}^n}$.

The corresponding Toeplitz operators with this kind of symbols are thus called
symmetric separately radial Toeplitz operators.

From Proposition \ref{prop1} we have that a symmetric separately radial Toeplitz operator is
$S_n\rtimes \mathbb{T}^n$-invariant in every Bergman space $\mathcal{H}^2_{\alpha}(\mathbb{B}^n)$.

Now, we consider $a$ a
symmetric separately radial function, hence, for each $m=\sigma(\iota)\in \mathcal{I}_\iota$ with $\sigma \in S_n$, using the fact that $a$ is separately radial it follows that
\begin{eqnarray*}
	\gamma_a(m)&=& \langle T_a(e_n), e_m \rangle_{\alpha} \\ &=& \langle a e_m, e_{m}\rangle_{\alpha}\\&=&  \frac{\Gamma(n+\vert m \vert +\alpha +1 ) }{m!\Gamma(n +\alpha +1 )} c_{\alpha}\int_{\mathbb{B}^n} a(z  ) z^m \bar{z}^{m}(1-\vert z \vert^2)^{\alpha} dv(z)\\
	&=&  \frac{n!\Gamma(n+\vert m \vert +\alpha +1 ) }{\pi^n m!\Gamma(n +\alpha +1 )} c_{\alpha}\int_{\mathbb{B}^n} a( z  ) \vert z \vert^{2m}(1-\vert z \vert^2)^{\alpha} dz
\end{eqnarray*}
and using change of variable $z=\sigma(w)$ and facts $\iota!=\sigma(\iota)!$ and $\vert \iota \vert=\vert \sigma(\iota) \vert$,  we obtain
\begin{eqnarray*}
	&=&  \frac{n!\Gamma(n+\vert \iota \vert +\alpha +1 ) }{\pi^n \iota!\Gamma(n +\alpha +1 )} c_{\alpha}\int_{\mathbb{B}^n} a( \sigma(w)  ) \vert \sigma(w) \vert^{2\sigma(\iota)}(1-\vert \sigma(w) \vert^2)^{\alpha} dw\\
	&=&  \frac{n!\Gamma(n+\vert \iota \vert +\alpha +1 ) }{\pi^n \iota!\Gamma(n +\alpha +1 )} c_{\alpha}\int_{\mathbb{B}^n} a(\vert w_1 \vert, \ldots, \vert w_n \vert ) \vert w \vert^{2\iota}(1-\vert w \vert^2)^{\alpha} dw
\end{eqnarray*}
and using polar coordinates in each axis  of $\mathbb{C}^n$ we obtain
$$=\frac{2^n\Gamma(n+\vert \iota \vert +\alpha +1 ) }{\iota!\Gamma(n +\alpha +1 )} \int_{\tau(\mathbb{B}^n)} a(r) r^{2\iota}(1-\vert r \vert^2)^{\alpha} \prod_{k=1}^n r_k dr_k.$$

In summary we have the following result.
\begin{theorem}{\label{Teor8}}
	The $C^*$-algebra generated by Toeplitz operators with symmetric separately radial symbols
	is commutative, i. e. , for every $\alpha > -1$  the Toeplitz operator  is  unitarily equivalent to
	$$RTR^*= \bigoplus_{\iota \in \mathcal{I}} \gamma_a(\iota) I_{\mid_{\ell^2(\mathcal{I}_{\iota})}} $$
	acting on
	$$\ell^2(\mathbb{Z}_+^n)=\bigoplus_{\iota\in \mathcal{I}}\ell^2(\mathcal{I}_{\iota}),$$
	where $R$ is given in Theorem \ref{Teor3} and  the spectral function is given by
	$$\gamma_a(\iota)=\frac{2^n\Gamma(n+\vert \iota \vert +\alpha +1 ) }{\iota!\Gamma(n +\alpha +1 )} \int_{\tau(\mathbb{B}^n)} a(r) r^{2\iota}(1-\vert r \vert^2)^{\alpha} \prod_{k=1}^n r_k dr_k $$
	where
	$\iota \in \mathcal{I}$.

	In particular, from the above result we obtain that  every Toeplitz operator with symmetric separately radial symbol
	is  unitarily equivalent to
	$$RTR^*= \gamma_a(m) I$$
	acting on $l^2(\mathbb{Z}_+^n)$ where
	$$\gamma_a(m)=\frac{2^n\Gamma(n+\vert m \vert +\alpha +1 ) }{m!\Gamma(n +\alpha +1 )} \int_{\tau(\mathbb{B}^n)} a(r) r^{2m}(1-\vert r \vert^2)^{\alpha} \prod_{k=1}^n r_k dr_k. $$
	Then the spectral function satisfies that
	$$\gamma_a(m)= \gamma_a(\iota)$$
	for all $m\in \mathcal{I}_{\iota}.$
\end{theorem}

\begin{remark}
	The spectral function associated to symmetric separately radial Toeplitz operators is invariant under permutations.
\end{remark}

\begin{remark}
	The symmetric separately radial Toeplitz operators are more general than radial Toeplitz operators, i.e., every radial Toeplitz operators is a symmetric separately radial.
\end{remark}
\begin{example}
	For $n=2$, we consider the symmetric separately radial function
	$$a_0(r_1,r_2)=r_1^2r_2^2$$
	The set of indices for the decomposition of the Bergman spaces $\mathcal{H}^2_{\alpha}(\mathbb{B}^2)$ associated to the group $S_2 \rtimes \mathbb{T}^2$ is given by
	$\mathcal{I}=\{(k_1,k_2): k_1\geq k_2\geq 0    \}.$
	The subspace of the Bergman spaces $\mathcal{H}^2_{\alpha}(\mathbb{B}^2)$ associated to $(k_1,k_2)\in \mathcal{I}$ is defined by
	$$\mathcal{P}_{(k_1,k_2)}(\mathbb{B}^n)=\left\{ \begin{array}{cc}
		\langle z_1^{k_1}z_2^{k_1}\rangle     & k_1=k_2  \\
		\langle z_1^{k_1}z_2^{k_2}, z_1^{k_2}z_2^{k_1} \rangle      & k_1 \neq k_2
	\end{array}  \right. $$
	where $\langle z_1^{k_1}z_2^{k_2}, z_1^{k_2}z_2^{k_1} \rangle$ is the space generated by these polynomials.
	Now, the spectral function of the Toeplitz operator with symbol $a_0$ is given by
	\begin{eqnarray*}
		\gamma_{a_0}(k_1,k_2)&=&\frac{2^2\Gamma(n+k_1+k_2 +\alpha +1 ) }{k_1!k_2!\Gamma(n +\alpha +1 )} \int_{\tau(\mathbb{B}^2)} r_1^{2k_1+3}r_2^{2k_2+3} (1-r_1^2-r_2^2)^{\alpha} dr_1 dr_2  \\
		& =&\frac{(k_1+1)(k_2+1) }{(n+k_1+k_2 +\alpha +2 )(n+k_1+k_2 +\alpha +1 )}.
	\end{eqnarray*}
	
	Note that if we consider $k_1,k_2,k_3 \in \mathbb{Z}_+$ with $k_1>k_2$ and $k_1+k_2=2k_3$ then $$\gamma_{a_0}(k_1,k_2)\neq\gamma_{a_0}(k_3,k_3)$$
	The above equation exemplifies the difference between the spectral functions associated with symmetric separately radial  and radial symbols respectively, since the spectral functions associated to a radial symbol are constant respect to the quantity $k_1+k_2$.
\end{example}

\begin{remark}
	It can be noted from the above example that the symmetric separately radial  symbols have a relationship with the symmetric polynomials, i.e, using the symmetric polynomials we can construct a family of symmetric separately radial functions which can be written as follows
	$$ a(z)=\hat{a}(P(\vert z_1\vert \ldots ,\vert z_n\vert))$$
	where $\hat{a}\in L^{\infty}(\mathbb{R}^n)$.
\end{remark}

\section{Toeplitz operators with alternating separately radial symbols}{\label{sec6}}

In the same way that we defined separately radial symmetric symbols in Section \ref{sec3}, we are going to define a new type of symbol.
We define an  alternating separately radial
function $a \in  L^{\infty}(\mathbb{B}^n)$ as one that can be written as follows
\begin{eqnarray}{\label{Symm:function}}
	a(z) = a(\vert z_1\vert, \ldots ,\vert z_n\vert)
\end{eqnarray}
and satisfies that
\begin{eqnarray}{\label{Alt:function}}
	a(\sigma(z))= a(z) \mbox{ where } \sigma \in A_n
\end{eqnarray}
for almost every $z \in \mathbb{B}^n$.

In other words, the function $a$ is alternating separately radial if and only if  it is $A_n\rtimes \mathbb{T}^n$-invariant.

The set of functions $A_n\rtimes \mathbb{T}^n$-invariant will be denoted by $\mathcal{A}_{A_n\rtimes \mathbb{T}^n}$.

The following result is the version of Theorem \ref{Teor8} in the case of a function which is a symbol
alternating separately radial, and its proof is completely analogous

\begin{theorem}
	The $C^*$-algebra generated by Toeplitz operators with alternating separately radial symbols
	is commutative, i. e. , for every $\alpha > -1$ the Toeplitz operator is  unitarily equivalent to
	
	$$RT_aR^*= \left(\bigoplus_{\iota \in \mathcal{I}^0} \gamma_a(\iota) I_{\mid_{\ell_2(\mathcal{I}_{\iota})}}\right)\bigoplus\left(\bigoplus_{\iota \in \mathcal{I}^c} \gamma_a(\iota) I_{\mid_{\ell(\mathcal{I}^+_{\iota})}}\oplus \gamma_a(\sigma(\iota)) I_{\mid_{\ell_2(\mathcal{I}^-_{\iota})}} \right),$$
	acting on
	$$ \ell^2(\mathbb{Z}_+^n)= \left(\bigoplus_{\iota \in \mathcal{I}^0} \ell^2(\mathcal{I}_{\iota})  \right)\bigoplus\left(\bigoplus_{\iota \in \mathcal{I}^c} \ell^2(\mathcal{I}^+_{\iota})\oplus \ell^2(\mathcal{I}^-_{\iota})\right),$$
	where $\sigma$ is odd and  the spectral function is given by
	$$\gamma_a(\iota)=\frac{2^n\Gamma(n+\vert \iota \vert +\alpha +1 ) }{\iota!\Gamma(n +\alpha +1 )} \int_{\tau(\mathbb{B}^n)} a(r) r^{2\iota}(1-\vert r \vert^2)^{\alpha} \prod_{k=1}^n r_k dr_k .$$
	and
	$$\gamma_a(\sigma(\iota))=\frac{2^n\Gamma(n+\vert \iota \vert +\alpha +1 ) }{\iota!\Gamma(n +\alpha +1 )} \int_{\tau(\mathbb{B}^n)} a(r) r^{2\sigma(\iota)}(1-\vert r \vert^2)^{\alpha} \prod_{k=1}^n r_k dr_k .$$
\end{theorem}

In particular, from the above result we obtain that for every Toeplitz operator with alternating separately radial symbol
is  unitarily equivalent to
$RT_aR^*= \gamma_a(m) I$
acting on $\ell^2(\mathbb{Z}_+^n)$, where

$$\gamma_a(m)=\frac{\Gamma(n+\vert m \vert +\alpha +1 ) }{m!\Gamma(n +\alpha +1 )} \int_{\tau(\mathbb{B}^n)} a(r) r^{2m}(1-\vert r \vert^2)^{\alpha} \prod_{k=1}^n r_k dr_k. $$

Then the spectral function satisfies that
\begin{enumerate}
	\item For $\iota \in \mathcal{I}^0$ we have
	$\gamma_a(\iota)=\gamma_a(m)$
	for every $m\in \mathcal{I}_{\iota}$, i.e.,
	$ \gamma_a(m)=\gamma_a(m')$ whenever  $m, m'\in \mathcal{I}_{\iota}$.
	\item For $\iota \in \mathcal{I}^c$ we have
	$\gamma_a(\iota)=\gamma_a(m)$
	for every $m\in \mathcal{I}_{\iota}^+$, i.e.,
	$ \gamma_a(m)=\gamma_a(m')$ whenever  $m, m'\in \mathcal{I}_{\iota}^+$.
	\item Consider $\sigma$ an odd permutation. For $\iota \in \mathcal{I}^c$  we have
	$\gamma_a(\sigma(\iota))=\gamma_a(m)$
	for every $m\in \mathcal{I}_{\iota}^-$, i.e.,
	$ \gamma_a(m)=\gamma_a(m')$ whenever  $m, m'\in \mathcal{I}_{\iota}^-$.
\end{enumerate}

\begin{remark}
	The alternating separately radial Toeplitz operators are more general than radial Toeplitz operators, i.e., every radial Toeplitz operator is alternating   separately radial.
	
We have the following chain of proper inclusions between the different sets of symbols defined in Section \ref{sec2}:
	$$ \mathcal{A}_{U(n)} \subsetneqq \mathcal{A}_{S_n \rtimes \mathbb{T}^n} \subsetneqq \mathcal{A}_{A_n \rtimes \mathbb{T}^n} \subsetneqq \mathcal{A}_{ \mathbb{T}^n}.  $$
	
	For the sake of completeness we will present explicit examples that prove some of the proper inclusions
	established previously.
	
	The symbol $a(r)= f(r_1)$ is an element in $\mathcal{A}_{ \mathbb{T}^n}$, but clearly does not belong to $\mathcal{A}_{ A_n \rtimes \mathbb{T}^n}$.
	
	The symbol $a(r_1,r_2,r_3) \in \mathcal{A}_{A_n \rtimes \mathbb{T}^n}$ given by:
	$$a(r_1,r_2,r_3)=f(r_1r_2^2 r_3^3-r_1^3r_2^2 r_3+r_2 r_3^2 r_1^3-r_2 r_2 r_3^3+r_3r_1^2 r_2^3-r_2r_1^2r_3^3   ),$$
	where $f$ is an odd function over $\mathbb{R}$,  does not belong to $\mathcal{A}_{S_n \rtimes \mathbb{T}^n}$, since
	$a(\sigma(r_1,r_2,r_3))= -a(r_1,r_2,r_3)$ when $\sigma$ is an odd permutation.

	We also have the following chain of proper inclusions between the different algebras generated by Toeplitz operators:
	$$ \mathcal{T}^{U(n)} \subsetneqq \mathcal{T}^{S_n \rtimes \mathbb{T}^{n}} \subsetneqq \mathcal{T}^{A_n \rtimes \mathbb{T}^{n}} \subsetneqq \mathcal{T}^{\mathbb{T}^{n}}.$$
\end{remark}

\section{Toeplitz operators with anti-symmetric separately radial symbols}{\label{sec7}}
Finally, we mention that what we proved for symmetric separately radial  and alternating separately radial  symbols also works for a new type of symbols: anti-symmetric separately radial symbols.

After enunciating the theorem that describes the explicit form of the spectral function for Toeplitz operators with antisymmetric symbols,
we will list a series of remarks about these symbols and the Toeplitz operators with such symbols.

We define an  anti-symmetric separately radial
function $a \in  L^{\infty}(\mathbb{B}^n)$ written as
\begin{eqnarray}\label{Symm:function}
	a(z) = a(\vert z_1\vert, \ldots ,\vert z_n\vert)
\end{eqnarray}
satisfying
\begin{eqnarray}\label{Anti:function}
	a(\sigma(z))= \mbox{Sgn}(\sigma)a(z) \mbox{ where } \sigma \in S_n
\end{eqnarray}
for almost every $z \in \mathbb{B}^n$.

In particular, every  anti-symmetric separately radial function is  alternating separately radial.

\begin{theorem}
	The $C^*$-algebra generated by Toeplitz operators with anti-symmetric separately radial symbols
	is commutative, i. e. , for every $\alpha > -1$ the Toeplitz operator is  unitarily equivalent to
	
	$$RT_aR^*= \left(\bigoplus_{\iota \in \mathcal{I}^0} \gamma_a(\iota) I_{\mid_{\ell^2(\mathcal{I}_{\iota})}}\right)\bigoplus\left(\bigoplus_{\iota \in \mathcal{I}^c} \gamma_a(\iota) I_{\mid_{\ell^2(\mathcal{I}^+_{\iota})}}\oplus \gamma_a(\sigma(\iota)) I_{\mid_{\ell^2(\mathcal{I}^-_{\iota})}} \right),$$
	acting on
	$$ \ell^2(\mathbb{Z}_+^n)= \left(\bigoplus_{\iota \in \mathcal{I}^0} \ell^2(\mathcal{I}_{\iota})  \right)\bigoplus\left(\bigoplus_{\iota \in \mathcal{I}^c} \ell^2(\mathcal{I}^+_{\iota})\oplus \ell^2(\mathcal{I}^-_{\iota})\right)$$
	where $\sigma$ is odd and  the spectral function satisfies that
	\begin{enumerate}
		\item  For $\iota \in \mathcal{I}^0$ we have $\gamma_a(\iota)=0$.
		\item  For $\iota \in \mathcal{I}^c$ we have $\gamma_a(\sigma(\iota))=-\gamma_a(\iota)$
	\end{enumerate}

	In particular, from the above result we conclude that every Toeplitz operator with anti-symmetric separately radial symbols
	is  unitarily equivalent to
	$RTR^*= \gamma_a(m) I$
	acting  on $\ell^2(\mathbb{Z}_+^n)$ where
	
	$$\gamma_a(m)=\frac{\Gamma(n+\vert m \vert +\alpha +1 ) }{m!\Gamma(n +\alpha +1 )} \int_{\tau(\mathbb{B}^n)} a(r) r^{2m}(1-\vert r \vert^2)^{\alpha} \prod_{k=1}^n r_k dr_k. $$
	
	Then the spectral function satisfies that
	\begin{enumerate}
		\item For $\iota \in \mathcal{I}^0$ we have
		$\gamma_a(m)=0$
		for every $m\in \mathcal{I}_{\iota}$.
		\item For $\iota \in \mathcal{I}^c$ we have
		$\gamma_a(\iota)=\gamma_a(m)$
		for every $m\in \mathcal{I}_{\iota}^+$.
		\item If $\sigma$ is an odd permutation, then for every $\iota \in \mathcal{I}^c$  we have
		$\gamma_a(m)=-\gamma_a(\iota)$
		for every $m\in \mathcal{I}_{\iota}^-$.
	\end{enumerate}
\end{theorem}

\begin{remark}
	
	If $a$ is an alternating  separately radial function and $\sigma$ an odd permutation, then  $a_{\sigma}(z)=a(\sigma(z))$ is  an alternating  separately radial function.
\end{remark}

\begin{remark} If $a$ is an alternating  separately radial function and $\sigma, \beta$ are odd permutations, then  $a(\sigma(z))=a(\beta(z))$ and $a(\beta \sigma(z))=a(z)$.
\end{remark}

\begin{remark}
	If $a$ is an alternating  separately radial function and $\sigma$ an odd permutation, then
	$$ a^+=\frac{ a+a_{\sigma} }{2}$$ is a symmetric function.
\end{remark}
\begin{remark}
	If $a$ is an alternating  separately radial function and $\sigma$ an odd permutation, then
	$$ a^-=\frac{ a-a_{\sigma} }{2}$$ is an anti-symmetric function.
\end{remark}

\begin{remark} If $a$ is an alternating  separately radial function and $\sigma$ an odd permutation, then
	$ a=a^+ + a^-$.
	
\end{remark}

\begin{remark}
	Toeplitz operators with alternating separately radial symbols $T_a=T_{a^+}+T_{a^-}.$ are unitarily equivalent to
	$$RT_{a^+}R^*= \left(\bigoplus_{\iota \in \mathcal{I}^0} \gamma_a(\iota) I_{\mid_{l^2(\mathcal{I}_{\iota})}}\right)\oplus \bigoplus_{\iota \in \mathcal{I}^c}  \gamma_a(\iota) \left(  I_{\mid_{l^2(\mathcal{I}^+_{\iota})}}\oplus I_{\mid_{l^2(\mathcal{I}^-_{\iota})}} \right) ,$$
	and
	$$RT_{a^-}R^*= \left(\bigoplus_{\iota \in \mathcal{I}^0} 0_{\mid_{l^2(\mathcal{I}_{\iota})}}\right)\oplus \bigoplus_{\iota \in \mathcal{I}^c}  \gamma_a(\iota) \left(  I_{\mid_{l^2(\mathcal{I}^+_{\iota})}}\oplus (-I)_{\mid_{l^2(\mathcal{I}^-_{\iota})}} \right) .$$
\end{remark}



\end{document}